\documentclass{svproc}
\usepackage{url}

\usepackage{mathptmx}       
\usepackage{helvet}         
\usepackage{courier}        
\usepackage{type1cm}        
\usepackage{makeidx}         
\usepackage{graphicx}        
\usepackage{multicol}        
\usepackage[bottom]{footmisc}
\usepackage{epstopdf}
\usepackage{amssymb}
\usepackage[centertags]{amsmath}
\usepackage{amsfonts}

\newtheorem{thm}{Theorem}[section]
\newtheorem{prop}[thm]{Proposition}
\newtheorem{dfn}[thm]{Definition}

\def\NN{{\mathbb{N}}}

\def\Acknowledgement{\goodbreak\bigskip\noindent{\bf Acknowledgement.\ }}

\begin{document}
\mainmatter              
\title{Superexponential stabilizability of degenerate parabolic equations via bilinear control}
\titlerunning{Superexponential stabilizability}  
%
\author{Piermarco Cannarsa\inst{1} \and Cristina Urbani\inst{2}}
\authorrunning{Piermarco Cannarsa et al.} 
%
\tocauthor{Piermarco Cannarsa and Cristina Urbani}
\institute{Univerist\`{a} di Roma Tor Vergata, 00133, Roma, Italy,\\
\email{cannarsa@mat.uniroma2.it},\\
\and
Gran Sasso Science Institute, 67100, L'Aquila, Italy,\\
Laboratoire Jacques-Louis Lions, Sorbonne Universit\'{e}, 75005, Paris, France,\\
\email{cristina.urbani@gssi.it}
}

\maketitle              

\begin{abstract}
The aim of this paper is to prove the superexponential stabilizability to the ground state solution of a degenerate parabolic equation of the form
\begin{equation*}
u_t(t,x)+(x^{\alpha}u_x(t,x))_x+p(t)x^{2-\alpha}u(t,x)=0,\qquad t\geq0,x\in(0,1)
\end{equation*}
via bilinear control $p\in L_{loc}^2(0,+\infty)$. More precisely, we provide a control function $p$ that steers the solution of the equation, $u$, to the ground state solution in small time with doubly-exponential rate of convergence.\\
The parameter $\alpha$ describes the degeneracy magnitude. In particular, for $\alpha\in[0,1)$ the problem is called weakly degenerate, while for $\alpha\in[1,2)$ strong degeneracy occurs. We are able to prove the aforementioned stabilization property for $\alpha\in [0,3/2)$. The proof relies on the application of an abstract result on rapid stabilizability of parabolic evolution equations by the action of bilinear control. A crucial role is also played by Bessel's functions.
\keywords{stabilization, bilinear control, degenerate equations, parabolic equations, Bessel's functions}
\end{abstract}
\section{Introduction}
The control of degenerate parabolic equations has received increasing attention by the mathematical community in recent years. In our opinion this fact is due to, at least, two reasons. First, degenerate parabolic operators occur in several applied contexts, such as population genetics \cite{cr,cc,e,ek}, fluids flows \cite{os}, and climate models \cite{d,dht,h}. Second, compared to uniformly parabolic problems, degenerate equations exhibit different behaviors from the point of view of controllability. Indeed, it is known that under the action of an additive control --- locally distributed or located at the boundary --- exact null controllability may fail if degeneracy is too violent, or else be true in any time $T>0$ (see \cite{cmvn}), or even be true after some critical time $T^*>0$, related to the distance of the control support from the degeneracy set, as proved, for instance, in \cite{bcg,bde,bmm}. 

In this paper, however, we are not interested in an additive control problem but rather in a bilinear one. More precisely, we investigate the response of the degenerate parabolic equation
\begin{equation}\label{1i}
\left\{
\begin{array}{ll}
u_t-\left(x^{\alpha} u_x\right)_x+p(t)\mu(x)u=0,& (t,x)\in (0,+\infty)\times(0,1)
\vspace{.1cm}\\
u(t,1)=0,\quad\left\{\begin{array}{ll} u(t,0)=0,& \mbox{ if }\alpha\in[0,1)\vspace{.1cm},\\
\left(x^{\alpha}u_x\right)(t,0)=0,& \mbox{ if }\alpha\in[1,2),\end{array}\right.
\vspace{.1cm}\\
u(0,x)=u_0(x).
\end{array}
\right.
\end{equation}
to the action of a scalar control $p\in L^2_{loc}(0,\infty)$. We observe that the importance of  bilinear control problems is due to the fact that they refer to materials that are able to react to  control inputs  by changing their principal parameters. This process is called \emph{catalysis} and it is described in some examples in \cite{k}. 

A stronger kind of control, which is intermediate --- in some sense --- between additive and bilinear control, is multiplicative control, where one uses a zero order coefficient, $p(t,x)$, to act upon the equation. In this direction, we recall the approximate controllability results by Khapalov et al. \cite{cfk,ck} for uniformly parabolic equations, and \cite{f} for degenerate parabolic models.

To understand the difference between bilinear and additive control it suffices to recall the celebrated negative result by Ball, Marsden and Slemrod~\cite{bms} for abstract evolution equations of the form 
\begin{equation}\label{2i}
\left\{
\begin{array}{ll}
u'(t)+Au(t)+p(t)Bu(t)=0,& t>0\vspace{.1cm}
\\
u(0)=u_0,
\end{array}\right.
\end{equation}
where $A$ is the infinitesimal generator of  a $C^0$-semigroup of bounded linear operators on a Banach space $X$, $B:X\to X$ is a bounded operator, and $p\in L^r_{loc}(0,\infty)$ for some $r>1$. Denoting the unique solution of \eqref{2i} by $u(\cdot;u_0,p)$, it was proved in \cite{bms} that the attainable set from $u_0$, defined by
\begin{equation*}
S(u_0)=\{ u(t;u_0,p)~:~t\geq 0,\; p\in L^r_{loc}(0,\infty)\},
\end{equation*}
has a dense complement. Therefore, \eqref{2i} fails to be controllable.

For hyperbolic and dispersive models, however, some positive results were later obtained. We would like to mention, in this respect, the results concerning attainable sets for the Schr{\"o}dinger and wave equations near the ground state solution, obtained in \cite{bl} and \cite{b}, respectively.

So, returning to the abstract problem \eqref{2i} for  a densely defined linear operator $A:D(A)\subset X\to X$, a natural question to investigate is the possibility of stabilizing the system near some specific solution.
We recall below a possible solution to such a problem in case  $X$ is a Hilbert space, which consists of the \emph{superexponential stabilizability}  property obtained in \cite{acu}  under the following assumptions:
\begin{equation}\label{ip}
\begin{array}{ll}
(a) & A \mbox{ is self-adjoint},
\vspace{.1cm}\\
(b) & A \mbox{ is accretive: }\langle Au,u\rangle \geq 0,\,\, \forall u\in D(A),
\vspace{.1cm}\\
(c) &\exists\,\lambda>0,\mbox{ such that }(\lambda I+A)^{-1}:X\to X \mbox{ is compact}.
\end{array}
\end{equation}
We denote  by $\{\lambda_k\}\,(0\leq \lambda_k\leq\lambda_{k+1})$ the eigenvalues of $A$ and by $\{\varphi_k\}$ the associated eigenvectors.
Recalling that $\varphi_1$ is usually called the {\em ground state} of $A$, we will refer to $\psi_1(t):=e^{-\lambda_1 t}\varphi_1$ as the {\em ground state solution} of \eqref{2i}  (with $p\equiv 0$). Finally, we denote by $B_R(u)$ the open ball of radius $R>0$, centered at $u\in X$.
\begin{thm}\label{t1}
Let $A:D(A)\subset X\to X$ be a linear operator on the Hilbert space $X$ satisfying hypothesis \eqref{ip}. Suppose that, for some  $\gamma>0$, 
\begin{equation}\label{gap}
\sqrt{\lambda_{k+1}}-\sqrt{\lambda_k}\geq \gamma,\quad\forall k\in \NN^*.
\end{equation}
Let $B: X\to X$ be a bounded linear operator with the following properties:
\begin{equation}\label{a2}
\begin{array}{l}
\langle B\varphi_1,\varphi_k\rangle\neq 0,\qquad\forall k\in\NN^*,
\vspace{.1cm}\\
\exists\, \tau>0\;\mbox{such that}\;
\displaystyle{\sum_{k\in\NN^*}\frac{e^{-2\lambda_k\tau}}{|\langle B\varphi_1,\varphi_k\rangle|^2}<\infty.}
\end{array}
\end{equation}
Then,  for every $\rho>0$ there exists $R>0$ such that  any $u_0\in B_R(\varphi_1)$ admits a control 
$p\in L^2_{loc}(0,\infty)$ such that the corresponding solution $u(\cdot;u_0,p)$ of \eqref{2i}  satisfies
\begin{equation}
||u(t)-\psi_1(t)||\leq M e^{-\rho e^{\omega t}-\lambda_1 t}\qquad \forall t\geq 0,
\end{equation}
where $M$ and $\omega$ are positive constants depending only on $A$ and $B$. 
\end{thm}

The purpose of this paper is to apply Theorem \ref{t1} to the degenerate control system \eqref{1i}, deducing local superexponential stabilizability for such a system.

From the technical point of view, we will have to check that operator $A$, given by the realization of the elliptic part of the equation in \eqref{1i}, and the multiplication operator $B$, associated to the coefficient $\mu(x)=x^{\alpha-1}$, satisfy the assumptions \eqref{ip}, \eqref{gap} and \eqref{a2}. For this purpose, the properties of Bessel's functions of the first kind will play a crucial role. Indeed, the eigenvalues and eigenfunctions of $A$ are related to such special functions and their zeros, as observed in \cite{cmvs,cmv,g}.

This paper is organized as follows. In section 2, we assemble preliminary material on degenerate parabolic equations and Bessel's functions. In section 3, we state and prove our main result.

\section{Preliminaries}
\label{sec:2}
Let $I=(0,1)$, $X=L^2(I)$ and consider the following degenerate parabolic equation
\begin{equation}\label{1}
\left\{
\begin{array}{ll}
u_t-\left(a(x)u_x\right)_x+p(t)\mu(x)u=0, &x\in I,\,\, t>0
\vspace{.1cm}
\\
u(0)=u_0,&x\in I
\end{array}
\right.
\end{equation}
where $p$ is the bilinear control function and $a(x)$ is the degenerate coefficient. Depending on the type of degeneracy, it is customary to assign different boundary conditions to the problem. 

Let us recall the definition of two different kinds of degenerate problems. Let 
\begin{equation}\label{a}
a\in C^0([0,1])\cap C^1((0,1]),\,\, a>0 \mbox{ on } (0,1] \mbox{ and } a(0)=0.
\end{equation}
Consider $u_0\in X$ and $p\in L^2_{loc}([0,\infty))$.
\begin{dfn}
If \eqref{a} holds and moreover
\begin{equation}
\frac{1}{a}\in L^1(I)
\end{equation}
we say that the controlled equation
\begin{equation}\label{5}
\left\{
\begin{array}{ll}
u_t-\left(a(x)u_x\right)_x+p(t)\mu(x)u=0, &x\in I,\,\, t>0\vspace{.1cm}
\\
u(t,0)=0,\,\,u(t,1)=0,&t>0\vspace{.1cm}
\\
u(0)=u_0,&x\in I.
\end{array}
\right.
\end{equation}
is \emph{weakly degenerate}.
\end{dfn}
\begin{dfn}
If \eqref{a} holds and moreover
\begin{equation}
a\in C^1([0,1])\mbox{ and }\frac{1}{\sqrt{a}}\in L^1(I)
\end{equation}
we say that the controlled equation
\begin{equation}\label{7}
\left\{
\begin{array}{ll}
u_t-\left(a(x)u_x\right)_x+p(t)\mu(x)u=0, &x\in I,\,\, t>0\vspace{.1cm}\\
(au_x)(t,0)=0,\,\,u(t,1)=0,&t>0\vspace{.1cm}\\
u(0)=u_0,&x\in I.
\end{array}
\right.
\end{equation}
is \emph{strongly degenerate}.
\end{dfn}
In particular, we will be interested in treating the degenerate coefficient $a(x)=x^{\alpha}$. Following the above definitions, we have a weakly degenerate problem for $\alpha\in[0,1)$ and a strongly degenerate one for $\alpha\in[1,2)$.

We will treat separately the cases of weak and strong degeneracy.

\subsection{Weak degeneracy}
Let $\alpha\in [0,1)$ and consider the degenerate bilinear control problem
\begin{equation}\label{8}
\left\{
\begin{array}{ll}
u_t-\left(x^{\alpha}u_x\right)_x+p(t)\mu(x)u=0, &x\in I,\,\, t>0\vspace{.1cm}\\
u(t,0)=0,\,\,u(t,1)=0,&t>0\vspace{.1cm}\\
u(0)=u_0,&x\in I
\end{array}
\right.
\end{equation}
with Dirichlet boundary conditions. The natural spaces for the well-posedness of degenerate problems are weighted Sobolev spaces. Let $X=L^2(I)$, we define the spaces
\begin{equation}
\begin{array}{l}
H^1_{\alpha}(I)=\left\{u\in X: u \mbox{ is absolutely continuous on } [0,1], x^{\alpha/2}u_x\in X\right\}\vspace{.1cm}\\
H^1_{\alpha,0}(I)=\left\{u\in H^1_\alpha(I):\,\, u(0)=0,\,\,u(1)=0\right\}\vspace{.1cm}\\
H^2_\alpha(I)=\left\{u\in H^1_\alpha(I): x^{\alpha}u_x\in H^1(I)\right\},
\end{array}
\end{equation}
and the linear operator $A:D(A)\subset X\to X$ by
\begin{equation}
\left\{\begin{array}{l}
\forall u\in D(A),\quad Au:=-(x^{\alpha}u_x)_x,\vspace{.1cm}\\
D(A):=\{u\in H^1_{\alpha,0}(I),\,\, x^{\alpha}u_x\in H^1(I)\}.
\end{array}\right.
\end{equation}
It is possible to prove that $D(A)$ is dense in X and $A:D(A)\subset X\to X$ is a self-adjoint accretive operator (see, for instance, \cite{cmp}). Therefore $-A$ is the infinitesimal generator of an analytic $C^0$-semigroup of contraction $e^{-tA}$ on $X$.

To determine the spectrum of $A$, we need to solve the eigenvalue problem 
\begin{equation}\label{11}
\left\{
\begin{array}{ll}
-(x^{\alpha}\varphi_x(x))_x=\lambda\varphi(x),&x\in I\vspace{.1cm}\\
\varphi(0)=0,\quad\varphi(1)=0,
\end{array}
\right.
\end{equation}
and it turns out that Bessel functions play a fundamental role in this circumstance.

Indeed, for $\alpha\in[0,1)$ let
\begin{equation}
\nu_\alpha:=\frac{1-\alpha}{2-\alpha},\qquad k_\alpha:=\frac{2-\alpha}{2}.
\end{equation}
Given $\nu\geq0$, we denote by $J_\nu$ the Bessel function of the first kind and order $\nu$ and by $j_{\nu,1}<j_{\nu,2}<\dots<j_{\nu,k}<\dots$ the sequence of all positive zeros of $J_\nu$. It is possible to prove that the pairs eigenvalue/eigenfunction $(\lambda_{\alpha,k},\varphi_{\alpha,k})$ that satisfy \eqref{11} are given by
\begin{equation}\label{13}
\lambda_{\alpha,k}=k^2_{\alpha}j^2_{\alpha,k},
\end{equation}
\begin{equation}\label{14}
\varphi_{\alpha,k}(x)=\frac{\sqrt{2k_\alpha}}{|J'_{\nu_\alpha}(j_{\nu_\alpha,k})|}x^{(1-\alpha)/2}J_{\nu_\alpha}\left(j_{\nu_\alpha,k}x^{k_\alpha}\right)
\end{equation}
for every $k\in\NN^*$. Moreover, the family $\left(\varphi_{\alpha,k}\right)_{k\in\NN^*}$ is an orthonormal basis of $X$, see \cite{g}.

\subsection{Strong degeneracy}
In the case of strong degeneracy, that is, when $\alpha\in[1,2)$, we consider the following degenerate bilinear control problem
 \begin{equation}
\left\{
\begin{array}{ll}
u_t-\left(x^{\alpha}u_x\right)_x+p(t)\mu(x)u=0, &x\in I,\,\, t>0\vspace{.1cm}\\
(x^{\alpha}u_x)(t,0)=0,\,\,u(t,1)=0,&t>0\vspace{.1cm}\\
u(0)=u_0,&x\in I
\end{array}
\right.
\end{equation}
with a Neumann condition at the extremum where degeneracy occurs, $x=0$, and a Dirichlet condition at $x=1$. 

We define the Sobolev spaces
\begin{equation}
\begin{array}{l}
H^1_{\alpha}(I)=\left\{u\in X: u \mbox{ is absolutely continuous on } (0,1],\,\, x^{\alpha/2}u_x\in X\right\}\vspace{.1cm}\\
H^1_{\alpha,0}(I):=\left\{u\in H^1_{\alpha}(I):\,\,u(1)=0\right\},\vspace{.1cm}\\
H^2_\alpha(I)=\left\{u\in H^1_\alpha(I):\,\, x^{\alpha}u_x\in H^1(I)\right\}
\end{array}
\end{equation}
and the linear operator $A:D(A)\subset X\to X$ by
\begin{equation}
\left\{\begin{array}{l}
\forall u\in D(A),\quad Au:=-(x^{\alpha}u_x)_x,\vspace{.1cm}\\
D(A):=\left\{u\in H^1_{\alpha,0}(I):\,\, x^{\alpha}u_x\in H^1(I)\right\}\vspace{.1cm}\\
\qquad\,\,\,\,\,=\left\{u\in X:\,\,u \mbox{ is absolutely continuous in (0,1] },\,\, x^{\alpha}u\in H^1_0(I),\right.\vspace{.1cm}\\
\qquad\qquad\,\,\,\left.x^{\alpha}u_x\in H^1(I)\mbox{ and } (x^{\alpha}u_x)(0)=0\right\}.
\end{array}\right.
\end{equation}
It can be proved that $D(A)$ is dense in $X$ and that $A$ is self-adjoint and accretive (see, for instance, \cite{cmvn}) and thus $-A$ is the infinitesimal generator of an analytic semigroup of contractions $e^{tA}$ on $X$.

To compute the eigenvalues and eigenfunctions of $A$, we should solve the eigenvalue problem
\begin{equation}\label{15}
\left\{
\begin{array}{ll}
-(x^{\alpha}\varphi_x(x))_x=\lambda\varphi(x),&x\in I\vspace{.1cm}\\
(x^\alpha\varphi_x)(0)=0,\vspace{.1cm}\\
\varphi(1)=0.
\end{array}
\right.
\end{equation}
For $\alpha\in[1,2)$, if we define the quantities
\begin{equation}
\nu_{\alpha}:=\frac{\alpha-1}{2-\alpha},\qquad k_\alpha:=\frac{2-\alpha}{2},
\end{equation}
the eigenvalues and eigenfunctions that solve \eqref{15} have the same structure as in the case of the weakly degenerate problem \eqref{13} and \eqref{14}.
Therefore, the family $(\varphi_{\alpha,k})_{k\in\NN^*}$ still forms an orthonormal basis of $X$.

\begin{prop}\label{prop}
Let $\alpha\in [1,2)$. The following properties holds true:
\begin{enumerate}
\item $|v(x)|\leq \frac{2||v||_{D(A)}}{\alpha-1}x^{1-\alpha},\, \forall\, v\in D(A)$,\vspace{.25cm}
\item $|x^{\alpha}v(x)|\leq C\sqrt{x},\, \forall\, v\in D(A)$,\vspace{.25cm}
\item for $\alpha\in[1,3/2)$ it holds that $$\displaystyle\lim_{x\to0}x^2v(x)w_x(x)=0, \quad\forall v,w\in D(A)$$,
\item for $\alpha\in[1,3/2)$ it holds that $$\displaystyle\lim_{x\to0}xv(x)w(x)=0,\quad\forall v,w\in D(A)$$,
\item let $\{\varphi_{\alpha,k}\}_{k\in\NN^*}$ be the family of eigenfunctions of $A$. For $\alpha\in[1,3/2)$ and for every $k,j\in\NN^*$, it holds that $$\displaystyle\lim_{x\to0}\left(x(\varphi_{\alpha,j})_x(x)\right)_xx^\alpha \varphi_{\alpha,k}(x)=0$$.
\end{enumerate}
\end{prop}
\begin{proof}
\begin{enumerate}
\item For all $v\in D(A)$ and $y\in I$, we have
\begin{equation*}
\begin{split}
|v(1)-v(y)|&=\left|\int_y^1v_x(x)dx\right|=\left|\int_y^1(x^{\alpha}v_x(x))\frac{1}{x^\alpha}dx\right|\\
&\leq \sup_{0<x<1}\left|x^{\alpha}v_x(x)\right|\frac{\left|1-y^{1-\alpha}\right|}{\alpha-1}\\
&\leq \frac{2||v||_{D(A)}}{\alpha-1}y^{1-\alpha}
\end{split}
\end{equation*}
where in the last inequality we have used that, for all $v\in D(A)$, it holds that
\begin{equation}\label{19}
|a(y) v_x(y)|=\left|\int_0^y \left(a(x) v_x(x)\right)_xdx\right|\leq||(av_x)_x||_X\sqrt{y}
\end{equation}
with $a(y)=y^{\alpha}$. Finally, recalling that $v(1)=0$, we obtain the desired formula.
\item For every $v\in D(A)$ and $y\in I$, we have
\begin{equation*}
|y^{\alpha}v(y)|\leq\left|\int_0^y (x^{\alpha}v(x))_xdx\right|\leq||(av)_x||_{X}\sqrt{y}.
\end{equation*}
\item Let $v,w\in D(A)$. We can rewrite $x^2v(x)w_x(x)$ as
\begin{equation}
x^{2-\alpha}v(x)x^\alpha w_x(x).
\end{equation}
Thanks to \eqref{19}, there exists a constant $C>0$ such that
\begin{equation}\label{21}
|x^\alpha w_x(x)|\leq Cx^{1/2}.
\end{equation}
Thus, using the first item and \eqref{21} we obtain that
\begin{equation}
|x^{2-\alpha}v(x)x^\alpha w_x(x)|\leq Cx^{2-\alpha}x^{1-\alpha}x^{1/2}
\end{equation}
and therefore the right-hand side tends to $0$ as $x$ goes to $0$ for $\alpha<3/2$.
\item Let $v\in D(A)$. It is sufficient to prove that $\displaystyle\lim_{x\to0}x^{1/2}v(x)=0$. 

For this purpose, we observe that the function $x^{1/2}v(x)$ is integrable in $I$: indeed, using again the first point of the Proposition, we get
\begin{equation*}
|x^{1/2}v(x)|\leq Cx^{1/2+1-\alpha}
\end{equation*}
that is integrable in $I$.
Moreover, the derivative of $x^{1/2}v(x)$ is integrable in $I$:
\begin{equation}
\left(x^{1/2}v(x)\right)_x=x^{1/2}v_x(x)+\frac{1}{2}x^{-1/2}v(x)
\end{equation}
and we can bound the two terms on the right by
\begin{equation*}
|x^{1/2}v_x(x)|\leq|x^{\alpha}v_x(x)x^{1/2-\alpha}|\leq Cx^{1-\alpha}
\end{equation*}
that is integrable for any $\alpha\in [1,2)$ and by
\begin{equation*}
|x^{-1/2}v(x)|\leq Cx^{1/2-\alpha}
\end{equation*}
that is integrable for $\alpha\in [1,3/2)$.

Thus, we can deduce that the function $x^{1/2}v(x)$ is absolutely continuous in $I$ for $\alpha\in [1,3/2)$. So, the limit
\begin{equation}
\displaystyle\lim_{x\to0^+}x^{1/2}v(x)=L
\end{equation}
does exist. If $L\neq0$, then $v(x)$ would be of the same order as $\frac{1}{x^{1/2}}$ near $0$. This contradicts the fact that $v\in X$. Thus, $L=0$.

\item Recalling that $(x^\alpha(\varphi_{\alpha,k})_x(x))_x=-\lambda_k\varphi_{\alpha,k}(x)$, we have
\begin{equation*}
\begin{split}
(x(\varphi_{\alpha,j})_x(x))_xx^{\alpha}\varphi_{\alpha,k}(x)&=(x^{\alpha}(\varphi_{\alpha,j})_x(x)x^{1-\alpha})_xx^{\alpha}\varphi_{\alpha,k}(x)\\
&=(x^{\alpha}(\varphi_{\alpha,j})_x(x))_xx^{1-\alpha}x^{\alpha}\varphi_{\alpha,k}(x)\\
&\quad+(1-\alpha)x^{\alpha}(\varphi_{\alpha,j})_x(x)x^{-\alpha}x^{\alpha}\varphi_{\alpha,k}(x)\\
&=-\lambda_jx\varphi_{\alpha,j}(x)\varphi_{\alpha,k}(x)+(1-\alpha)x^{\alpha}(\varphi_{\alpha,j})_x(x)\varphi_{\alpha,k}(x).
\end{split}
\end{equation*}
The first of the two terms in the last equation on the right-hand side of the above formula goes to $0$ as $x\to0$, for $\alpha<3/2$, by the previous item. Moreover, we have
\begin{equation*}
|x^{\alpha}(\varphi_{\alpha,j})_x(x)\varphi_{\alpha,k}(x)|\leq Cx^{1/2}x^{1-\alpha}.
\end{equation*}
Therefore,
\begin{equation*}
\displaystyle\lim_{x\to0}\left(x(\varphi_{\alpha,j})_x(x)\right)_xx^\alpha \varphi_{\alpha,k}(x)=0
\end{equation*}
for $\alpha\in[1,3/2)$, as it was claimed.
\end{enumerate}
\end{proof}

\section{Main result}
\begin{thm}
Let $\alpha\in[0,3/2)$. 
Then,  for every $\rho>0$ there exists $R>0$ such that  any $u_0\in B_R(\varphi_1)$ admits a control 
$p\in L^2_{loc}(0,\infty)$ such that the corresponding mild solution $u\in C([0,1];X)$   
of
\begin{equation}\label{28}
\left\{
\begin{array}{ll}
u_t-\left(x^{\alpha} u_x\right)_x+p(t)x^{2-\alpha}u=0,& (t,x)\in (0,\infty)\times(0,1)\vspace{.1cm}\\
u(t,1)=0,\quad\left\{\begin{array}{ll} u(t,0)=0,& \mbox{ if }\alpha\in[0,1),\vspace{.1cm}\\ \left(x^{\alpha}u_x\right)(t,0)=0,& \mbox{ if }\alpha\in[1,3/2),\end{array}\right.\vspace{.1cm}\\
u(0,x)=u_0(x).
\end{array}
\right.
\end{equation}
satisfies
\begin{equation}
||u(t)-\psi_1(t)||\leq M e^{-\rho e^{\omega t}-\lambda_1 t},\qquad \forall t\geq 0
\end{equation}
where $M$ and $\omega$ are positive constants depending only on $\alpha$. 
\end{thm}

\begin{proof}
The proof of the Theorem consists in checking the validity of the hypotheses of Theorem \ref{t1}. We have already observed that $D(A)$ is dense in $X$ and that $A:D(A)\subseteq X\to X$ is self-adjoint and accretive, in both weakly and strongly degenerate cases. Moreover, it can be proved that $A$ has a compact resolvent (see, for instance, \cite[Appendix]{acf})

Concerning the gap conditions for the eigenvalues \eqref{gap}, it has been proved (see \cite{kl}, page 135) that
\begin{itemize}
\item if $\alpha\in [0,1)$, $\nu_{\alpha}=\frac{1-\alpha}{2-\alpha}\in\left(0,\frac{1}{2}\right]$, the sequence $\left( j_{\nu_\alpha,k+1}-j_{\nu_\alpha,k}\right)_{k\in\NN^*}$ is nondecreasing and converges to $\pi$. Therefore, 
\begin{equation*}
\sqrt{\lambda_{k+1}}-\sqrt{\lambda_k}=k_{\alpha}\left( j_{\nu_\alpha,k+1}-j_{\nu_\alpha,k}\right)\geq k_{\alpha}\left( j_{\nu_\alpha,2}-j_{\nu_\alpha,1}\right)\geq \frac{7}{16}\pi,
\end{equation*}
\item if $\nu_{\alpha}\geq \frac{1}{2}$, the sequence $\left( j_{\nu_\alpha,k+1}-j_{\nu_\alpha,k}\right)_{k\in\NN^*}$ is nonincreasing and converges to $\pi$. Thus,
\begin{equation*}
\sqrt{\lambda_{k+1}}-\sqrt{\lambda_k}=k_{\alpha}\left( j_{\nu_\alpha,k+1}-j_{\nu_\alpha,k}\right)\geq k_{\alpha}\pi\geq \frac{\pi}{2}.
\end{equation*}
\end{itemize}
Therefore, the gap condition is satisfied in both weak and strong degenerate problems with different constants.

The operator $B:X\to X$ is the multiplication operator by the function $\mu(x)=x^{2-\alpha}$ and it is linear and bounded in $I$. What remains to prove in order to apply Theorem \ref{t1}, is that there exists $\tau>0$ such that
\begin{equation}\label{32b}
\begin{array}{l}
\langle \mu\varphi_{\alpha,1},\varphi_{\alpha,k}\rangle\neq0,\qquad \forall k\in\NN^*,\vspace{.1cm}\\
\displaystyle{\sum_{k\in\NN^*}\frac{e^{-2\lambda_k\tau}}{|\langle \mu\varphi_{\alpha,1},\varphi_{\alpha,k}\rangle|^2}<+\infty.}
\end{array}
\end{equation}

We compute the scalar product $\langle \mu\varphi_{\alpha,1},\varphi_{\alpha,k}\rangle$ for $k\not=1$ and, from now on, we write $\varphi_k$ instead of $\varphi_{\alpha,k}$ to lighten the notation:
\begin{equation}\begin{split}\label{32}
\langle \mu\varphi_1,\varphi_k\rangle&=\int_0^1 \mu(x)\varphi_1(x)\varphi_k(x)=-\frac{1}{\lambda_k}\int_0^1\mu(x)\varphi_1(x)\left(x^\alpha(\varphi_k)_x(x)\right)_xdx\\
&=-\frac{1}{\lambda_k}\left(\left.\mu(x)\varphi_1(x)x^\alpha(\varphi_k)_x(x)\right|^1_0-\int_0^1\left(\mu(x)\varphi_1(x)\right)_xx^\alpha(\varphi_k)_x(x)dx\right)\\
&=\frac{1}{\lambda_k}\left(\int_0^1 \mu_x(x)\varphi_1(x)x^\alpha(\varphi_k)_x(x)dx+\int_0^1\mu(x)(\varphi_1)_x(x)x^\alpha(\varphi_k)_x(x)dx\right)\\
&=\frac{1}{\lambda_k}\left(\int_0^1 \mu_x(x)\varphi_1(x)x^\alpha(\varphi_k)_x(x)dx+\left.\mu(x)(\varphi_1)_x(x)x^\alpha\varphi_k(x)\right|^1_0\right.\\
&\qquad\qquad\qquad\qquad\qquad\qquad\left.-\int_0^1\left(\mu(x)(\varphi_1)_x(x)x^\alpha\right)_x\varphi_k(x)dx\right)\\
&=\frac{1}{\lambda_k}\left(\int_0^1 \mu_x(x)\varphi_1(x)x^\alpha(\varphi_k)_x(x)dx-\int^1_0\mu_x(x)(\varphi_1)_x(x)x^\alpha\varphi_k(x)dx\right.\\
&\,\qquad\qquad\qquad\qquad\qquad\qquad\left.-\int_0^1\mu(x)(x^\alpha(\varphi_1)_x(x))_x\varphi_k(x)dx\right)\\
&=\frac{1}{\lambda_k}\left(\int_0^1 \mu_x(x)x^\alpha\left[\varphi_1(x)(\varphi_k)_x(x)-(\varphi_1)_x(x)\varphi_k(x)\right]dx\right.\\
&\,\,\quad\qquad\qquad\qquad\qquad\qquad\qquad\left.+\lambda_1\int_0^1\mu(x)\varphi_1(x)\varphi_k(x)dx\right).
\end{split}
\end{equation}
We observe that in the weakly degenerate case, thanks to the Dirichlet conditions in both extrema, the boundary terms vanish. We can deduce the same vanishing property at $x=0$ for the strong degenerate case thanks to the first item of Proposition \ref{prop} and to \eqref{19}.

Moving the last term of \eqref{32} to the left-hand side, we get
\begin{equation}
\left(1-\frac{\lambda_1}{\lambda_k}\right)\langle \mu\varphi_1,\varphi_k\rangle=\frac{1}{\lambda_k}\int_0^1 \mu_x(x)x^\alpha\varphi_1^2(x)\left(\frac{\varphi_k(x)}{\varphi_1(x)}\right)_xdx
\end{equation}
and therefore, integrating by parts we obtain
\begin{equation}
\begin{split}
\langle \mu\varphi_1,\varphi_k\rangle&=\frac{1}{\lambda_k-\lambda_1}\left(\left.\mu_x(x)x^\alpha\varphi^2_1(x)\frac{\varphi_k(x)}{\varphi_1(x)}\right|^1_0-\int_0^1\left(\mu_x(x)x^\alpha\varphi^2_1(x)\right)_x\frac{\varphi_k(x)}{\varphi_1(x)}dx\right)\\
&=\frac{-1}{\lambda_k-\lambda_1}\left(\int_0^1(\mu_x(x)x^\alpha)_x\varphi^2_1(x)\frac{\varphi_k(x)}{\varphi_1(x)}dx\right.\\
&\,\,\,\,\,\,\qquad\qquad\qquad\qquad\qquad\qquad\left.+2\int_0^1\mu_x(x)x^\alpha\varphi_1(x)(\varphi_1)_x(x)\frac{\varphi_k(x)}{\varphi_1(x)}dx\right)\\
&=-\frac{1}{\lambda_k-\lambda_1}\left(\int_0^1(\mu_x(x)x^\alpha)_x\varphi_1(x)\varphi_k(x)dx\right.\\
&\left.\,\,\quad\qquad\qquad\qquad\qquad\qquad\qquad\qquad+2\int_0^1\mu_x(x)x^\alpha(\varphi_1)_x(x)\varphi_k(x)dx\right).
\end{split}
\end{equation}
The boundary terms vanish for the Dirichet conditions if $\alpha\in [0,1)$ and thanks to the second item in Proposition \ref{prop} for $\alpha\in [1,3/2)$.

Recalling that $\mu(x)=x^{2-\alpha}$, we have that
\begin{equation}\label{35}
\begin{split}
\langle \mu\varphi_1,\varphi_k\rangle&=-\frac{2(2-\alpha)}{\lambda_k-\lambda_1}\int_0^1 x(\varphi_1)_x(x)\varphi_k(x)dx\\
&=\frac{2(2-\alpha)}{\lambda_k(\lambda_k-\lambda_1)}\int_0^1x(\varphi_1)_x(x)\left(x^\alpha(\varphi_k)_x(x)\right)_xdx\\
&=\frac{2(2-\alpha)}{\lambda_k(\lambda_k-\lambda_1)}\left(\left.x(\varphi_1)_x(x)x^\alpha(\varphi_k)_x(x)\right|^1_0-\int_0^1\left(x(\varphi_1)_x(x)\right)_xx^\alpha(\varphi_k)_x(x)dx\right)\\
&=\frac{2(2-\alpha)}{\lambda_k(\lambda_k-\lambda_1)}\left(\left.x^{1+\alpha}(\varphi_1)_x(x)(\varphi_k)_x(x)\right|^1_0-\left.\left(x(\varphi_1)_x(x)\right)_xx^\alpha\varphi_k(x)\right|^1_0\right.\\
&\,\,\quad\qquad\qquad\qquad\qquad\qquad\qquad\left.+\int_0^1\left(\left(x(\varphi_1)_x(x)\right)_xx^\alpha\right)_x\varphi_k(x)dx\right)\\
&=\frac{2(2-\alpha)}{\lambda_k(\lambda_k-\lambda_1)}\left(\left.x^{1+\alpha}(\varphi_1)_x(x)(\varphi_k)_x(x)\right|^1_0\right.\\
&\left.\,\quad\qquad\qquad\qquad\qquad+\int_0^1\left(\left((\varphi_1)_x(x)+x(\varphi_1)_{xx}(x)\right)x^\alpha\right)_x\varphi_k(x)dx\right)\\
&=\frac{2(2-\alpha)}{\lambda_k(\lambda_k-\lambda_1)}\left(\left.x^{1+\alpha}(\varphi_1)_x(x)(\varphi_k)_x(x)\right|^1_0-\lambda_1\int_0^1\varphi_1(x)\varphi_k(x)dx\right.\\
&\quad\qquad\qquad\qquad\qquad\qquad\qquad\left.+\int_0^1\left(x^{1+\alpha}(\varphi_1)_{xx}(x)\right)_x\varphi_k(x)dx\right)\\
&=\frac{2(2-\alpha)}{\lambda_k(\lambda_k-\lambda_1)}\left(\left.x^{1+\alpha}(\varphi_1)_x(x)(\varphi_k)_x(x)\right|^1_0+\int_0^1\left(x^{1+\alpha}(\varphi_1)_{xx}(x)\right)_x\varphi_k(x)dx\right)
\end{split}
\end{equation}
where we have used the fact that, for $\alpha\in[1,3/2)$, $\left.(x(\varphi_1)_x(x))_xx^{\alpha}\varphi_k(x)\right|^1_0$ vanishes in view of Proposition \ref{prop}.

Since $\varphi_k$ is an eigenfunction of $A$ for all $k\in\NN^*$, it satisfies the equation
\begin{equation}\label{36}
-(\alpha x^{\alpha-1}(\varphi_k)_x(x)+x^\alpha (\varphi_k)_{xx}(x))=\lambda_k\varphi_k(x),
\end{equation}
then we can rewrite the expression of $(\varphi_k)_{xx}(x)$ in \eqref{35} using \eqref{36}:
\begin{equation}\label{37}
\begin{split}
\langle \mu\varphi_1,\varphi_k\rangle&=\frac{2(2-\alpha)}{\lambda_k(\lambda_k-\lambda_1)}\left(\left.x^{1+\alpha}(\varphi_1)_x(x)(\varphi_k)_x(x)\right|^1_0+\int_0^1\left(x^{1+\alpha}(\varphi_1)_{xx}(x)\right)_x\varphi_k(x)dx\right)\\
&=\frac{2(2-\alpha)}{\lambda_k(\lambda_k-\lambda_1)}\left(\left.x^{1+\alpha}(\varphi_1)_x(x)(\varphi_k)_x(x)\right|^1_0\right.\\
&\left.\qquad\qquad\qquad\qquad\qquad-\int_0^1\left(\lambda_1x\varphi_1(x)+\alpha x^\alpha(\varphi_1)_x(x)\right)_x\varphi_k(x)dx\right)\\
&=\frac{2(2-\alpha)}{\lambda_k(\lambda_k-\lambda_1)}\left(\left.x^{1+\alpha}(\varphi_1)_x(x)(\varphi_k)_x(x)\right|^1_0-\lambda_1\int_0^1x(\varphi_1)_x(x)\varphi_k(x)dx\right.\\
&\quad-\lambda_1\int_0^1\varphi_1(x)\varphi_k(x)dx-\alpha\int_0^1\underbrace{(x^\alpha(\varphi_1)_x(x))_x}_{-\lambda_1\varphi_1(x)}\varphi_k(x)dx\Big).
\end{split}
\end{equation}
Recalling that $\{\varphi_k\}_{k\in\NN^*}$ is an orthonormal base of $L^2(0,1)$, the last two terms on the right-hand side of the above equality are zero.

Thus, from the first equality of \eqref{35} and the last one of \eqref{37}, we obtain that
\begin{equation}
-\frac{2(2-\alpha)}{\lambda_k-\lambda_1}\left(1-\frac{\lambda_1}{\lambda_k}\right)\int_0^1 x(\varphi_1)_x(x)\varphi_k(x)dx=\frac{2(2-\alpha)}{\lambda_k(\lambda_k-\lambda_1)}\left.x^{1+\alpha}(\varphi_1)_x(x)(\varphi_k)_x(x)\right|^1_0
\end{equation}
that implies
\begin{equation}\label{39}
\langle \mu\varphi_1,\varphi_k\rangle=-\frac{2(2-\alpha)}{\lambda_k-\lambda_1}\int_0^1 x(\varphi_1)_x(x)\varphi_k(x)dx=\frac{2(2-\alpha)}{(\lambda_k-\lambda_1)^2}\left.x^{1+\alpha}(\varphi_1)_x(x)(\varphi_k)_x(x)\right|^1_0
\end{equation}

Recalling that the eigenvalues $\{\lambda_k\}_{k\in\NN^*}$ of $A$ are defined by \eqref{13} where $\nu_\alpha=|1-\alpha|/(2-\alpha)$, and the eigenfunctions, $\{\varphi_k\}_{k\in\NN^*}$, by \eqref{14}, we compute the right-hand side of \eqref{39}:
\begin{equation}\begin{split}
x^{1+\alpha}&(\varphi_1)_x(x)(\varphi_k)_x(x)=\\
&=\frac{2(2-\alpha)k_{\alpha}x^{1+\alpha}}{|J'_{\nu_\alpha}(j_{\nu_\alpha,1})||J'_{\nu_\alpha}(j_{\nu_\alpha,k})|}\left(\frac{1-\alpha}{2}x^{-(1+\alpha)/2}J_{\nu_\alpha}(j_{\nu_\alpha,1}x^{k_\alpha})\right.\\
&\qquad\qquad\qquad\qquad\qquad\qquad\qquad+j_{\nu_\alpha,1}k_\alpha x^{(1-2\alpha)/2}J'_{\nu_\alpha}(j_{\nu_\alpha,1}x^{k_\alpha})\Big)\\
&\quad\cdot\left(\frac{1-\alpha}{2}x^{-(1+\alpha)/2}J_{\nu_\alpha}(j_{\nu_\alpha,k}x^{k_\alpha})+j_{\nu_\alpha,k}k_\alpha x^{(1-2\alpha)/2}J'_{\nu_\alpha}(j_{\nu_\alpha,k}x^{k_\alpha})\right).
\end{split}
\end{equation}
Therefore
\begin{equation}\label{43}
x^{1+\alpha}(\varphi_1)_x(x)(\varphi_k)_x(x)\left.\right|^1_0=(\varphi_1)_x(1)(\varphi_k)_x(1)=\frac{2k_{\alpha}^3j_{\nu_\alpha,1}j_{\nu_\alpha,k}}{|J'_{\nu_\alpha}(j_{\nu_\alpha,1})||J'_{\nu_\alpha}(j_{\nu_\alpha,k})|}J'_{\nu_\alpha}(j_{\nu_\alpha,1})J'_{\nu_\alpha}(j_{\nu_\alpha,k}).
\end{equation}
Now, recall that the zeros of $J'_{\nu_\alpha}$, $j'_{\nu_\alpha,k}$, satisfy $\nu_\alpha<j'_{\nu_\alpha,1}<j_{\nu_\alpha,1}<j'_{\nu_\alpha,2}<j_{\nu_\alpha,2}\dots$, to conclude that the right-hand side of \eqref{43} does not vanish.

From \eqref{39} and \eqref{43} we deduce that there exists a constant $C$ such that
\begin{equation}\label{44}
\left| \langle \mu\varphi_1,\varphi_k\rangle\right|\geq \frac{C}{\lambda_k^{3/2}},\quad\forall k\in\NN^*,\,k\neq1.
\end{equation}

For $k=1$, we have
\begin{equation}
\begin{split}
\langle \mu\varphi_1,\varphi_1\rangle&=\frac{2k_\alpha}{|J'_{\nu_\alpha}
(j_{\nu_\alpha,1})|^2}\int_0^1 x^{2-\alpha}x^{1-\alpha}J^2_{\nu_\alpha}(j_{\nu_\alpha,1}x^{k_\alpha})dx\\
&=\frac{4k_\alpha j^4_{\nu_\alpha,1}}{(2-\alpha)|J'_{\nu_\alpha}
(j_{\nu_\alpha,1})|^2}\int_0^{j_{\nu_\alpha,1}}z^3J^2_{\nu_\alpha}(z)dz.
\end{split}
\end{equation}

We now appeal to the identity
\begin{equation}\label{46}
\begin{split}
(\sigma+&2)\int^z t^{\sigma+2}J_{\nu}(t)dt=(\sigma+1)\left\{\nu^2-\frac{1}{4}(\sigma+1)^2\right\}\int^z t^\sigma J^2_\nu(t)dt\\
&+\frac{1}{2}z^{\sigma+1}\left[\left\{zJ'_\nu(z)-\frac{1}{2}(\sigma+1)J_\nu(z)\right\}^2+\left\{z^2-\nu^2+\frac{1}{4}(\sigma+1)^2\right\}J^2_\nu(z)\right]
\end{split}
\end{equation}
with $\sigma=1$ (see \cite{l}, equation (17) page 256) to turn \eqref{46} into
\begin{equation}\label{47}
\begin{split}
\langle \mu\varphi_1,&\varphi_1\rangle=\frac{4k_\alpha j^4_{\nu_\alpha,1}}{(2-\alpha)|J'_{\nu_\alpha}
(j_{\nu_\alpha,1})|^2}\frac{2}{3}\left\{\nu_\alpha^2-1\right\}\int_0^{j_{\nu_\alpha,1}}zJ^2_{\nu_\alpha}(z)dz\\
&+\frac{1}{6}j^3_{\nu_\alpha,1}\left[\left\{j_{\nu_\alpha,1}J'_{\nu_\alpha}(j_{\nu_\alpha,1})-J_{\nu_\alpha}(j_{\nu_\alpha,1})\right\}^2+\left\{j^2_{\nu_\alpha,1}-\nu_\alpha^2+1\right\}J^2_{\nu_\alpha}(j_{\nu_\alpha,1})\right].
\end{split}
\end{equation}
Using Lommel's integral
\begin{equation}
\int_0^c zJ_\nu(az)^2dz=\frac{c^2}{2}\left[J^2_\nu(ac)-J_{\nu-1}(ac)J_{\nu+1}(ac)\right]
\end{equation}
in \eqref{47}, we obtain
\begin{small}\begin{equation*}
\begin{split}
\langle \mu\varphi_1,&\varphi_1\rangle=\\
&=\frac{4k_\alpha j^4_{\nu_\alpha,1}}{(2-\alpha)|J'_{\nu_\alpha}
(j_{\nu_\alpha,1})|^2}\left[\frac{2}{3}\left\{\nu_\alpha^2-1\right\}\left(\frac{j^2_{\nu_\alpha,1}}{2}\left(J^2_{\nu_\alpha}(j_{\nu_\alpha,1})-J_{\nu_\alpha-1}(j_{\nu_\alpha,1})J_{\nu_\alpha+1}(j_{\nu_\alpha,1})\right)\right)\right.\\
&\,\qquad\qquad\qquad\qquad\qquad\qquad\qquad\qquad\qquad\qquad\qquad\quad+\left.\frac{1}{6}j^3_{\nu_\alpha,1}\left(j_{\nu_\alpha,1}J'_{\nu_\alpha}(j_{\nu_\alpha,1})\right)^2\right]\\
&=\frac{4k_\alpha j^4_{\nu_\alpha,1}}{(2-\alpha)|J'_{\nu_\alpha}
(j_{\nu_\alpha,1})|^2}\left(-\frac{1}{3}j^2_{\nu_\alpha,1}\left\{\nu_\alpha^2-1\right\}J_{\nu_\alpha-1}(j_{\nu_\alpha,1})J_{\nu_\alpha+1}(j_{\nu_\alpha,1})\right.\\
&\,\,\,\,\quad\qquad\qquad\qquad\qquad\qquad\qquad\qquad\qquad+\left.\frac{1}{24}j^5_{\nu_\alpha,1}\left(J_{\nu_\alpha-1}(j_{\nu_\alpha,1})-J_{\nu_\alpha+1}(j_{\nu_\alpha,1})\right)^2\right)\\
&=\frac{4k_\alpha j^4_{\nu_\alpha,1}}{(2-\alpha)|J'_{\nu_\alpha}
(j_{\nu_\alpha,1})|^2}\left(\frac{1}{24}j^5_{\nu_\alpha,1}\left(J^2_{\nu_\alpha-1}(j_{\nu_\alpha,1})+J^2_{\nu_\alpha+1}(j_{\nu_\alpha,1})\right)\right.\\
&\,\,\,\quad\qquad\qquad\qquad\qquad\qquad-\left.\left(\frac{1}{3}j^2_{\nu_\alpha,1}\left\{\nu_\alpha^2-1\right\}+\frac{j^5_{\nu_\alpha,1}}{12}\right)J_{\nu_\alpha-1}(j_{\nu_\alpha,1})J_{\nu_\alpha+1}(j_{\nu_\alpha,1})\right)\\
&\geq \frac{4k_\alpha j^4_{\nu_\alpha,1}}{(2-\alpha)|J'_{\nu_\alpha}
(j_{\nu_\alpha,1})|^2}\left(\frac{1}{24}j^5_{\nu_\alpha,1}\left(J^2_{\nu_\alpha-1}(j_{\nu_\alpha,1})+J^2_{\nu_\alpha+1}(j_{\nu_\alpha,1})\right)\right.\\
&\,\,\,\,\,\quad\quad\qquad\qquad\qquad\left.-\left(\frac{1}{3}j^2_{\nu_\alpha,1}\left\{\nu_\alpha^2-1\right\}+\frac{j^5_{\nu_\alpha,1}}{12}\right)\frac{1}{2}\left(J^2_{\nu_\alpha-1}(j_{\nu_\alpha,1})+J^2_{\nu_\alpha+1}(j_{\nu_\alpha,1})\right)\right).
\end{split}
\end{equation*}\end{small}

Thus, $\langle \mu\varphi_1,\varphi_1\rangle>0$ if
\begin{equation}\label{50}
\frac{1}{24}j^5_{\nu_\alpha,1}>\frac{1}{2}\left(\frac{1}{3}j^2_{\nu_\alpha,1}\left\{\nu_\alpha^2-1\right\}+\frac{j^5_{\nu_\alpha,1}}{12}\right).
\end{equation}
Since 
\begin{equation*}
\begin{array}{l}
\alpha\in[0,1)\,\,\Rightarrow\,\, \nu_\alpha\in(0,1/2],\vspace{.1cm}\\
\alpha\in[1,3/2)\,\,\Rightarrow \nu_\alpha\in[0,1),
\end{array}
\end{equation*}
equation \eqref{50} holds true for both weak and strong degeneracy.

Thus, since $\langle \mu\varphi_1,\varphi_k\rangle\neq0$ for every $k\in\NN^*$ and \eqref{44} is valid, the series \eqref{32b} converges for every $\tau>0$.

We have checked that every hypothesis of Theorem \ref{t1} holds for problem \eqref{28} if $\alpha\in[0,3/2)$. Therefore, we conclude that, for any $\rho>0$, if the initial condition $u_0$ is close enough to $\varphi_1$, the system is superexponentially stabilizable to the ground state solution $\psi_1$. Moreover, the following estimate holds true
\begin{equation}
||u(t)-\psi_1(t)||\leq M e^{-(\rho e^{\omega t}+\lambda_1 t)},\qquad \forall t\geq 0.
\end{equation}
where $M,\omega>0$ are suitable constants.
\end{proof}

\Acknowledgement{
This paper was partly supported by the INdAM National Group for Mathematical Analysis, Probability and their Applications.
The first author acknowledges support from the MIUR Excellence Department Project awarded to the Department of Mathematics, University of Rome Tor Vergata, CUP E83C18000100006.  The second  author is grateful to University Italo Francese (Vinci Project 2018) for partial support.}

%
%


\begin{thebibliography}{50}

%
\bibitem {acf}
Alabau-Boussouira, F., Cannarsa, P. and Fragnelli, G.: Carleman estimates for degenerate parabolic operators with applications to null controllability. Journal of Evolution Equations, vol. 6, n. 2, pages 161--204 (2006), Springer

\bibitem {acu}
Alabau-Boussouira, F., Cannarsa, P. and Urbani, C.: Superexponential stabilizability of evolution equations of parabolic type via bilinear control. Preprint available on arXiv:1910.06802

\bibitem {bms}
Ball, J.M., Marsden, J.E. and Slemrod, M.: Controllability for distributed bilinear systems. SIAM Journal on Control and Optimization, vol. 20, n. 4, pages 575--597 (1982), SIAM

\bibitem {b}
Beauchard, K.: Local controllability and non-controllability for a 1D wave equation with bilinear control. Journal of Differential Equations, vol. 250, n. 4, pages 2064--2098 (2011), Academic Press

\bibitem {bcg}
Beauchard, K., Cannarsa, P. and Guglielmi, R.: Null controllability of {G}rushin-type operators in dimension two. Journal of the European Mathematical Society, vol. 16, n. 1, pages 67--101 (2014), Europeran Mathematical Society

\bibitem {bl}
Beauchard, K.  and Laurent, C.: Local controllability of 1D linear and nonlinear {S}chr{\"o}dinger equations with bilinear control. J. Math. Pures Appl., vol. 94, n. 5, pages 520--554 (2010)

\bibitem {bmm}
Beauchard, K., Miller, L. and Morancey, M.: 2D {G}rushin-type equations: minimal time and null controllable data. Journal of Differential Equations, vol. 259, n. 11, pages 5813--5845 (2015), Elsevier

\bibitem {bde}
Beauchard, K., Dard{\'e}, J. and Ervedoza, S.: Minimal time issues for the observability of {G}rushin-type equations. hal-01677037 (2018)

\bibitem {cmp}
Campiti, M., Metafune, G. and Pallara, D.: Degenerate self-adjoint evolution equations on the unit interval. Semigroup Forum, vol. 57, n. 1, pages 1--36 (1998), Springer

\bibitem {cr}
Campiti, M. and Rasa, I.: Qualitative properties of a class of {F}leming-{V}iot operators. Acta Mathematica Hungarica, vol. 103, n. 1-2, pages 55--69 (2004), Akad{\'e}miai Kiad{\'o}, co-published with Springer Science+ Business Media BV~…

\bibitem {cfk}
Cannarsa, P., Floridia, G. and Khapalov, A. Y.: Multiplicative controllability for semilinear reaction-diffusion equations with finitely many changes of sign. Journal de Math{\'e}matiques Pures et Appliqu{\'e}es, vol. 108, n. 4, pages 425--458 (2017), Elsevier

\bibitem{ck}
Cannarsa, P. and Khapalov, A.Y.: Multiplicative controllability for reaction-diffusion equations with target states admitting finitely many changes of sign. Discrete Contin. Dyn. Syst. Ser. B., vol. 14, pages 1293--1311 (2010), Citiseer

\bibitem {cmvn}
Cannarsa, P., Martinez, P. and Vancostenoble, J.: Carleman estimate for a class of degenerate parabolic operators. SIAM Journal on Control and Optimization, vol. 47, n. 1, pages 1--19 (2008), SIAM

\bibitem {cmvs}
Cannarsa, P., Martinez, P. and Vancostenoble, J.: The cost of controlling strongly degenerate parabolic equations. Journal of the European Mathematical Society, vol.16, n. 1, pages 67--101 (2013)

\bibitem {cmv}
Cannarsa, P., Martinez, P. and Vancostenoble, J.: The cost of controlling weakly degenerate parabolic equations by boundary controls. Mathematical Control \& Related Fields, vol. 7, n. 2, pages 71--211 (2017)

\bibitem {cc}
Cerrai, S. and Cl{\'e}ment, P.: On a class of degenerate elliptic operators arising from {F}leming-{V}iot processes. Journal of Evolution Equations, vol. 1, n. 3, pages 243--276 (2001), Springer

\bibitem {d}
D{\'\i}az, J. I.: On the mathematical treatment of energy balance climate models. The Mathematics of Models for Climatology and Environment, pages. 217--251 (1997), Springer

\bibitem {dht}
Diaz, J.I., Hetzer, G. and Tello, L.: An energy balance climate model with hysteresis. Nonlinear Analysis: Theory, Methods \& Applications, vol. 64, n. 9, pages 2053--2074 (2006), Elsevier

\bibitem {e}
Ethier, S. N.: A class of degenerate diffusion processes occurring in population genetics. Communications on Pure and Applied Mathematics, vol. 29, n. 5, pages 483--493 (1976), Wiley Online Library

\bibitem {ek}
Ethier, S. N. and Kurtz, T. G: {F}leming--{V}iot processes in population genetics. SIAM Journal on Control and Optimization, vol. 31, n. 2, pages 345--386 (1993), SIAM

\bibitem {f}
Floridia, G.: Approximate controllability for nonlinear degenerate parabolic problems with bilinear control. Journal of Differential Equations, vol. 257, n. 9, pages 3382--3422 (2014), Elsevier

\bibitem {g}
Gueye, M.: Exact boundary controllability of 1-D parabolic and hyperbolic degenerate equations. SIAM Journal on Control and Optimization, vol. 52, n. 4, pages 2037--2054 (2014), SIAM

\bibitem {h}
Hetzer, G.: The number of stationary solutions for a one-dimensional {B}udyko-type climate model. Nonlinear Analysis: Real World Applications, vol. 2, n. 2, pages 259--272 (2001)

\bibitem {k}
Khapalov, A.Y.: Controllability of partial differential equations governed by multiplicative controls (2010), Springer

\bibitem {kl}
Komornik, V. and Loreti, P.: Fourier series in control theory (2005), Springer Science \& Business Media

\bibitem {l}
Luke, Y. L.: Integrals of {B}essel functions (2014), Courier Corporation

\bibitem {os}
Oleinik, O. A. and Samokhin, V. N.: Mathematical models in boundary layer theory. Vol. 15 (1999), CRC Press

%
%
%
%
\end{thebibliography}
\end{document}